\documentclass[11pt,a4paper]{article}

\usepackage{tikz}
\usepackage{subfigure}
\usepackage[english]{babel}
\usepackage{bbm}
\usepackage{graphicx}
\usepackage[center]{caption2}
\usepackage{amsfonts,amssymb,amsmath,latexsym,amsthm}
\usepackage{multirow}
\usepackage{enumerate}
\usepackage{geometry}
\usepackage[usenames,dvipsnames]{pstricks}
\usepackage{epsfig}
\usepackage{pst-grad} 
\usepackage{pst-plot} 
\usepackage[space]{grffile} 
\usepackage{etoolbox} 
\makeatletter 
\patchcmd\Gread@eps{\@inputcheck#1 }{\@inputcheck"#1"\relax}{}{}

\def\kn{\mathrm{KNRS}}

\newtheorem{myDef}{Definition}
\newtheorem{thm}{Theorem}[section]
\newtheorem{cor}[thm]{Corollary}
\newtheorem{lem}[thm]{Lemma}

\newtheorem{conj}[thm]{Conjecture}

\newtheorem{claim}{Claim}

\let\svthefootnote\thefootnote
\newcommand\blankfootnote[1]{%
	\let\thefootnote\relax\footnotetext{#1}%
	\let\thefootnote\svthefootnote%
}

\begin{document}

\title{\Large Kohayakawa-Nagle-R{\"o}dl-Schacht conjecture for subdivisions}
	
\date{}
	
\author{
Hao Chen
~~~~~~~
Yupeng Lin
~~~~~~~
Jie Ma
}

\maketitle
	
\begin{abstract}
In this paper, we study the well-known Kohayakawa-Nagle-R{\"o}dl-Schacht (KNRS) conjecture,
with a specific focus on graph subdivisions.
The KNRS conjecture asserts that for any graph $H$, locally dense graphs contain asymptotically at least the number of copies of $H$ found in a random graph with the same edge density.
We prove the following results about $k$-subdivisions of graphs (obtained by replacing edges with paths of length $k+1$).
\begin{itemize}
    \item If $H$ satisfies the KNRS conjecture, then its $(2k-1)$-subdivision satisfies Sidorenko's conjecture, extending a prior result of Conlon, Kim, Lee and Lee;
    \item If $H$ satisfies the KNRS conjecture, then its $2k$-subdivision satisfies a constant-fraction version of the KNRS conjecture;
    \item If $H$ is regular and satisfies the KNRS conjecture, then its $2k$-subdivision also satisfies the KNRS conjecture.
\end{itemize}
These findings imply that all balanced subdivisions of cliques satisfy the KNRS conjecture,
improving upon a recent result of Brada{\v c}, Sudakov and Wigerson.
Our work provides new insights into this pivotal conjecture in extremal graph theory.
\blankfootnote{School of Mathematical Sciences, University of Science and Technology of China, Hefei, Anhui, 230026, China. Research supported by National Key Research and Development Program of China 2023YFA1010201 and National Natural Science Foundation of China grant 12125106.}
\end{abstract}

\section{Introduction}
Let $H$ and $G$ be two graphs.
A {\it graph homomorphism} from $H$ to $G$ is a map $f:V(H)\rightarrow V(G)$ such that $f(u)f(v)\in E(G)$ whenever $uv\in E(H)$.
The {\it homomorphism density} $t(H,G)$ denotes the probability that a random map from $V(H)$ to $V(G)$ forms a graph homomorphism.
Throughout this paper, the {\it $k$-subdivision} of $H$, denoted by $H^{(k)}$, is obtained by replacing each edge $uv$ of $H$ with an internally disjoint path of length $k+1$ between $u$ and $v$.

\subsection{Background}
One of the most fundamental problems in Combinatorics is to accurately estimate the occurrence of a specific combinatorial substructure.
In this paper, we delve into the problem of counting the occurrences of a fixed subgraph in graphs with a given number of vertices and edges.
In this direction, a well-known conjecture by Sidorenko~\cite{S93} states that for any bipartite graph $H$, the random graph minimizes $t(H,G)$ among all graphs $G$ with a fixed edge density.

\begin{conj}[Sidorenko's conjecture~\cite{S93}]\label{Sidorenko}
For every bipartite graph $H$, it holds that $t(H,G)\geq t(K_2,G)^{|E(H)|}$ for every graph $G$.
\end{conj}

If a bipartite graph $H$ satisfies this conjecture, then we say $H$ is {\it Sidorenko}.
Despite being proven for various specific bipartite graphs, Sidorenko's conjecture remains widely open (see e.g.,~\cite{CKLL18,CL17,CL21} for some recent advancements in this conjecture).
The smallest open case of Sidorenko's conjecture is $K_{5,5} \setminus C_{10}$, a 3-regular bipartite graph on 10 vertices.

If $H$ is non-bipartite, then the aforementioned inequality cannot hold in general.
This is because we can take counterexamples $G$ to be any bipartite graph, resulting in $t(H,G)=0$ while $t(K_2,G)>0$.
It is natural to wonder if, by adding some extra condition on the host graph $G$, Conjecture~\ref{Sidorenko} could be extended to general graphs. Kohayakawa, Nagle, R{\"o}dl and Schacht~\cite{KNRS10} speculate that the only barrier for a given non-bipartite graph $H$ to satisfy Sidorenko's conjecture is the presence of a large sparse vertex subset within the host graph $G$.
Following~\cite{KNRS10}, an $n$-vertex graph $G$ is called {\it $(\rho,d)$-dense} if any subset $X$ of at least $\rho n$ vertices of $G$ contains at least $\frac{d}{2}|X|^2$ edges.
We also call these graphs {\it locally dense}.
Kohayakawa et al.~\cite{KNRS10} posed a question which is now widely recognized as a conjecture.

\begin{conj}[Kohayakawa, Nagle, R{\"{o}}dl and Schacht~\cite{KNRS10}]\label{knrs}
For every graph $H$ and all reals $d, \varepsilon>0$, there exist a $\rho=\rho(d,H,\varepsilon)>0$ such that $t(H,G)\geq (1-\varepsilon)\cdot d^{|E(H)|}$ holds for every sufficiently large $(\rho,d)$-dense graph $G$.
\end{conj}

We say a graph $H$ is {\it $\kn$} if $H$ satisfies Conjecture~\ref{knrs}.
Observe that any bipartite graph satisfying Conjecture~\ref{Sidorenko} evidently satisfies Conjecture~\ref{knrs},
i.e., any Sidorenko graph is $\kn$.
On the other hand, only very few non-bipartite graphs are known to be $\kn$.
These known examples include all complete multipartite graphs~\cite{KNRS10}, all odd cycles~\cite{Re14}, unicycle graphs~\cite{L21}, all graphs obtained from a cycle by adding a chord~\cite{L21}, as well as graphs obtained by gluing complete multipartite graphs or odd cycles in a tree-like way~\cite{L21}.

Very recently, Brada{\v c}, Sudakov and Wigderson~\cite{BSW} proved that graphs constructed from smaller KNRS graphs by certain symmetric gluing operations are KNRS.
Furthermore, Brada{\v c}, Sudakov and Wigderson~\cite{BSW} also explored a more restricted version of Conjecture~\ref{knrs} by imposing an ``almost regular'' condition on the host graph.
An $n$-vertex graph $G$ is called {\it $(\rho,d)$-nearly-regular} if the degrees of all but at most $\rho n$ vertices of $G$ fall within the range $[(d-\rho)n,(d+\rho)n]$.
Brada{\v c} et al. made the following conjecture.

\begin{conj}[Brada{\v c}, Sudakov and Wigderson \cite{BSW}, see Conjecture~1.8]\label{regular-KNRS}
For every graph $H$ and all real $d,\varepsilon>0$, there exists a $\rho=\rho(d,H,\varepsilon)>0$ such that $t(H,G)\geq (1-\varepsilon)\cdot d^{|E(H)|}$ holds for every sufficiently large $(\rho,d)$-dense $(\rho,d)$-nearly-regular graph $G$.
\end{conj}

Following the notation in \cite{BSW} we say a graph $H$ is {\it regular-$\kn$} if $H$ satisfies Conjecture~\ref{regular-KNRS}.
By employing the theory of positive semidefinite matrices,
Brada{\v c}, Sudakov and Wigderson (\cite{BSW}, Theorem~1.9) verified the following graphs that exhibit the regular-$\kn$ property:
\begin{itemize}
\item[(a)] the $\ell$-subdivision of any regular-KNRS graph for any integer $\ell\geq 1$,
\item[(b)] all graphs obtained from a regular-KNRS graph by gluing forests to its vertices,
\item[(c)] all generalized theta graphs, and
\item[(d)] the graph with vertex set $\mathbb{Z}_6$ and edge set $\{(i,i+1):1\leq i\leq 6\}\cup \{(1,5),(2,4)\}$.\footnote{This is the smallest open case of Conjecture~\ref{knrs} (see~\cite{L21} for more discussions).}
\end{itemize}
As pointed out by the authors of \cite{BSW}, item (a) implies that all balanced subdivisions of cliques are regular-KNRS.
In an earlier but related result, Conlon, Kim, Lee and Lee~\cite{CKLL18} showed that if a graph $H$ is $\kn$, then the $1$-subdivision of $H$ is Sidorenko.


\subsection{Weakly-KNRS}
Before presenting our results, we introduce a novel property that serves as an intermediate concept between the $\kn$ property and the regular-$\kn$ property.

\begin{myDef}[Weakly-KNRS graphs]
A graph $H$ is weakly-KNRS, if there exists some constant $c_H\in (0,1]$ such that for every real $d>0$, there exists a $\rho=\rho(d,H)>0$ such that  $t(H,G)\geq c_H\cdot d^{|E(H)|}$ holds for every sufficiently large $(\rho,d)$-dense graph $G$.
\end{myDef}
It is obvious that if a graph $H$ is KNRS, then $H$ is weakly-KNRS.
On the other hand, by utilizing the so-called tensor power trick, one can demonstrate that if $H$ is weakly-KNRS, then $H$ is regular-KNRS.
This can be formally derived from the following equivalent formulation of Proposition 4.10 of \cite{BSW}: For any graph $H$, if there exists an absolute constant $c>0$ such that for every real $d>0$, there exists a $\rho=\rho(d,H,c)>0$ satisfying $t(H,G)\geq c\cdot d^{|E(H)|}$ for every sufficiently large $(\rho,d)$-dense $(\rho,d)$-nearly-regular graph $G$, then the same inequality holds with $c=1$.

We refer to Figure 1 for an illustration of the relationship between Sidorenko graphs, KNRS graphs, weakly-KNRS graphs, and regular-KNRS graphs.

\begin{figure}[htbp]
\centering
\includegraphics[width=0.8\textwidth]{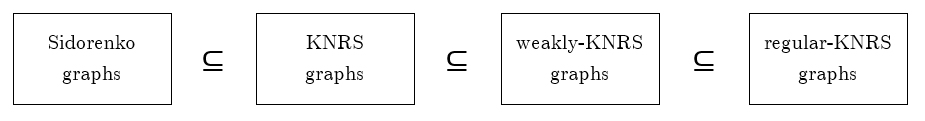}
\caption{An illustration among families of graphs}
\label{Figure 1:}
\end{figure}

\subsection{Our results}
In this paper, we mainly investigate the above conjectures for subdivisions of graphs.
For a comprehensive comparison with our findings, we recall two aforementioned results:
\begin{itemize}
\item Brada{\v c} et al. \cite{BSW} proved that if $H$ is regular-KNRS, then $H^{(\ell)}$ is also regular-KNRS;
\item Conlon et al. \cite{CKLL18} proved that if $H$ is $\kn$, then its $1$-subdivision $H^{(1)}$ is Sidorenko.
\end{itemize}
Note that repeatedly applying the above result of Conlon et al. \cite{CKLL18}, it can be inferred easily that for any $k\geq 1$,
if $H$ is $\kn$, then its $(2^k-1)$-subdivision $H^{(2^k-1)}$ is Sidorenko.

Now, we proceed to present our first result.
Assuming that $H$ satisfies the $\kn$ property, we strengthen the above result of Brada{\v c} et al. \cite{BSW}, and extend the result of Conlon et al. \cite{CKLL18} from $(2^k-1)$-subdivisions to $\ell$-subdivisions for all odd integers $\ell$.

\begin{thm}\label{even sub}
Let $k\geq 1$ be integers. If a graph $H$ is $\kn$, the following statements hold:
\begin{itemize}
\item[(a)] the $(2k-1)$-subdivision $H^{(2k-1)}$ is Sidorenko, and
\item[(b)] the $2k$-subdivision $H^{(2k)}$ is weakly-KNRS.
\end{itemize}
\end{thm}

We would like to point out that from the proof of this result, it can be shown that every balanced subdivision of a weakly-KNRS graph is also weakly-KNRS.

Our second result establishes that if one imposes an additional condition that $H$ is a regular graph,
then its $2k$-subdivision $H^{(2k)}$ becomes $\kn$.

\begin{thm}\label{regular sub}
Let $k\geq 1$ be integers. If a graph $H$ is $\kn$ and regular, then its $2k$-subdivision $H^{(2k)}$ is also $\kn$.
\end{thm}

It is important to highlight the notable distinction between this result and item (a) of Theorem 1.9 in \cite{BSW}.
In the former, we impose the regular condition on the graph $H$, whereas in item (a), the (almost) regular condition applies to the host graph $G$.

Using both Theorems~\ref{even sub} and \ref{regular sub}, we can readily derive the following corollary.

\begin{cor}\label{balance sub}
If a graph $H$ is $\kn$ and regular, then all balanced subdivisions of $H$ are $\kn$.
In particular, all balanced subdivisions of cliques are $\kn$.
\end{cor}

This improves a recent result of Brada{\v c}, Sudakov and Wigderson  \cite{BSW}, stating that all balanced subdivisions of cliques are regular-$\kn$.

This paper is organized as follows.
In Section 2, we provide necessary notations and some auxiliary lemmas.
In Section 3, we present the proofs of Theorems~\ref{even sub} and \ref{regular sub}.
Finally, in Section 4, we conclude with some remarks.

\section{Preliminaries}
In this section, we introduce terminologies and prepare some auxiliary lemmas used throughout the paper.
We start with some basic notions of graphs.
Let $G$ be a graph. The vertex set of $G$ is denoted by $V(G)$ and the edge set of $G$ is denoted by $E(G)$.
Let $v(G)$ and $e(G)$ be the number of vertices and edges of $G$, respectively.
Let $k$ be a positive integer.
The path with $k$ edges is denoted by $P_k$ and the complete graph with $k$ vertices is denoted by $K_k$.

In line with current trends and for the sake of notational convenience,
the subsequent section presents essential analytic terminologies needed in this paper.
All integrations in this paper are taken with respect to the Lebesgue measure and when we use the notion of measure we always mean the Lebesgue measure on $\mathbb{R}^n$.
Let $A$ be a measurable set of $\mathbb{R}^n$. We write $|A|$ for the measure of $A$.
If $|A|=0$, then we call $A$ a {\it zero-measure} set.
For a measurable function $f:[0,1]\rightarrow \mathbb{R}$, the $L_1$ norm $||f||_1$ of $f$ is defined by $||f||_1:=\int_{[0,1]}|f(x)|dx$.

\subsection{Graph limits}
In this paper, we work within the framework of graph limits, drawing upon the language and concepts associated with this field.
(For a comprehensive introduction to graph limits, we refer readers to~\cite{LL12}.)
A {\it kernel} is a measurable function $U:[0,1]^2\rightarrow \mathbb{R}$ satisfying $U(x,y)=U(y,x)$ for all $x,y\in [0,1]$.
A {\it graphon} is a kernel $W:[0,1]^2\rightarrow [0,1]$.
Roughly speaking, a graphon can be regarded as the limit of regularity partitions that approximate large graphs.
Consequently, numerative graphic parameters can be naturally extended to quantities defined on graphons.

As a graphon analogy of the vertex degree in a graph, the {\it degree function} of $x\in [0,1]$ in a graphon $W$ is defined as $${\rm deg}_W(x):=\int_{[0,1]}W(x,y)dy.$$
If ${\rm deg}_W(x)=d$ for all $x\in [0,1]$ except for a zero-measure subset $X\subseteq [0,1]$, then we say that $W$ is {\it $d$-regular}.
The {\it homomorphism density} of a graph $H$ in a graphon $W$ is defined as
$$t(H,W):=\int_{[0,1]^{v(H)}}{\prod_{ij\in E(H)}U(x_i,x_j)\prod_{v\in V(H)}dx_{v}}.$$
Specially, we call $t(K_2,W)$ the {\it density} of the graphon $W$.
Now we can restate Sidorenko's conjecture in the language of graph limits (see e.g. \cite{LL12}, Chapter 7 for details).

\begin{conj}[Sidorenko's conjecture, graphon form]\label{Sido limits}
Let $H$ be a bipartite graph. Then $t(H,W)\geq t(K_2,W)^{e(H)}$ holds for any graphon $W$.
\end{conj}

It can be shown that any graph $H$ satisfies Conjecture~\ref{Sidorenko} if and only if $H$ satisfies Conjecture~\ref{Sido limits}.
The next lemma shows that Sidorenko's conjecture can be further reduced to the case when the host graphon is regular.

\begin{lem}[\cite{CR21}, Theorem 8.2] \label{Sido-regular}
Let $H$ be a bipartite graph. Then $H$ is Sidorenko if and only if $t(H,W)\geq d^{e(H)}$ holds for every $d\in [0,1]$ and every $d$-regular graphon $W$.
\end{lem}

We now formalize the definition of locally dense graphs in the language of graph limits, as outlined in \cite{BSW}.

\begin{myDef}
A graphon $W$ on $[0,1]^2$ is called $d$-locally dense if for every measurable set $S\subseteq [0,1]$,
we have that $\int_{S\times S}W(x,y)dxdy\geq d|S|^2.$
\end{myDef}

The next lemma, which is proved in~\cite{BSW}, translates the statement of Conjecture~\ref{knrs} from locally dense graphs to locally dense graphons.
\begin{lem}[\cite{BSW}, Lemma 2.6]\label{graphon form}
Let $H$ be a graph. Then $H$ is KNRS if and only if $t(H,W)\geq d^{e(H)}$ holds for every $d$ and every $d$-locally dense graphon $W$.
\end{lem}

The next lemma follows directly from the proof of the above lemma, which we omit here.

\begin{lem}\label{weakly-knrs-limits}
Let $H$ be a graph. Then $H$ is weakly-KNRS if and only if there exists some constant $c_H\in (0,1]$ such that $t(H,W)\geq c_H\cdot d^{e(H)}$ holds for every $d$ and every $d$-locally dense graphon $W$.
\end{lem}

\subsection{The restricting graphon on measurable closed sets}
To establish the graphon analogy of a subgraph in a graph,
we require the following notion, which can be found in \cite{KNNVW22,KVW} for further details.
Fix a graphon $W$ and a measurable closed subset $A\subseteq [0,1]$ with $|A|>0$.
Let $\mathbbm{1}_A$ be the indicator function of $A$, i.e., $\mathbbm{1}_A(x)=1$ if $x\in A$ and $\mathbbm{1}_A(x)=0$ otherwise. Let $F:[0,|A|]\rightarrow [0,1]$ be the measurable function defined as
 \begin{equation}\label{f(z)}
 F(z):= \{t\in [0,1]: \int_{[0,t]}\mathbbm{1}_A(x)dx= z\}.
 \end{equation}
Then we have $F(z)\in A$ for any $z\in [0,|A|]$, and $\int_B \mathbbm{1}_A(x)dx=|F^{-1}(B)|$ holds for every measurable subset $B\subseteq [0,1]$.
Define the graphon $W[A]$ by setting $$W[A](x,y):=W(F(x\cdot|A|),F(y\cdot|A|))$$ for all $(x,y)\in [0,1]^2$.
We call $W[A]$ the {\it restricting graphon} of $W$ on $A$.
Observe that
\begin{equation}\label{W[A]'}
t(H,W[A])=\frac{1}{|A|^{v(H)}}\int_{[0,1]^{v(H)}}\prod_{ij\in E(H)}W(x_i,x_j)\prod_{v\in V(H)}\mathbbm{1}_A(x_v)dx_v
\end{equation}
holds for every graph $H$ and every measurable closed set $A\subseteq [0,1]$ with $|A|>0$.
Moreover, it implies
\begin{equation}\label{W[A]}
t(H,W)\geq |A|^{v(H)}t(H,W[A]).
\end{equation}

In the remaining of this subsection, we aim to prove Lemma~\ref{sub locally dense} that if $W$ is $d$-locally dense,
then the restricting graphon of $W$ on any measurable closed subset is also $d$-locally dense.
To show this, we need some more notion.
Given a measurable closed set $A\subseteq [0,1]$, define the function $F(\cdot)$ by \eqref{f(z)}.
For any measurable subset $B'$ in $[0,1]$, define the subset
\begin{equation}\label{equ:B}
B=\{y\in [0,1]:~ y=F(x\cdot|A|) \text{ for some }x\in B'\}.
\end{equation}
By the definition of $F$, we see that $B$ is a measurable subset of $A$ and moreover, $|F^{-1}(B)|=\int_B\mathbbm{1}_A(x)dx=|B|$.
It then follows that $|B'|=\frac{|F^{-1}(B)|}{|A|}=\frac{|B|}{|A|}.$
Before proving Lemma~\ref{sub locally dense}, we need the coming result, whose proof can be found in Appendix~\ref{app}.

\begin{lem}\label{F(B|A|)}
Given a graphon $W$ and measurable subsets $A, B, B'$ as above, if $A$ is closed, then it holds that
$$\int_{B'\times B'}W[A](x,y)dxdy=\frac{1}{|A|^2}\int_{B\times B}W(x,y)dxdy.$$
\end{lem}

Now we are ready to present the main result of this subsection.
\begin{lem}\label{sub locally dense}
If $W$ is a $d$-locally dense graphon, then for any measurable closed set $A\subseteq [0,1]$, the restricting graphon $W[A]$ is also $d$-locally dense.
\end{lem}

\begin{proof}
The case when $A$ is a zero-measure set is trivial, so we may assume that $|A|>0$.
Let $B'$ be any measurable subset in $[0,1]$ and define the subset $B\subseteq A$ by \eqref{equ:B}.
Using Lemma~\ref{F(B|A|)},
\begin{equation}
\begin{aligned}
\int_{B'\times B'}W[A](x,y)dxdy&=\frac{1}{|A|^2}\int_{B\times B}W(x,y)dxdy\geq \frac{d|B|^2}{|A|^2}=d|B'|^2,
\end{aligned}
\end{equation}
where the inequality holds because $W$ is $d$-locally dense and the last equality follows from the fact $|B'|=\frac{|B|}{|A|}$.
This implies that $W[A]$ is $d$-locally dense.
\end{proof}

\subsection{Reiher's inequality}
The following lemma proved by Reiher~\cite{Re14} has been proven to be extremely useful when considering Conjecture~\ref{knrs}.

\begin{lem}[Reiher's inequality~\cite{Re14}, Lemma 2.1]
Let $G$ be a $(\rho,d)$-dense graph with $n$ vertices. Let $f:V(G)\rightarrow [0,1]$ be a function satisfying that $\sum_{v\in V(G)}f(v)\geq \rho n$. Then $$\sum_{xy\in E(G)}f(x)f(y)\geq \frac{d}{2}\left(\sum_{x\in V(G)}f(x)\right)^2-n.$$
\end{lem}

We need following lemmas from \cite{BSW}, which extend Reiher's lemma to the context of graphons.

\begin{lem}[\cite{BSW}, Lemma 2.8]\label{Reiher}
Let $W$ be a $d$-locally dense graphon. Then for every bounded measurable function $f:[0,1]\rightarrow [0,\infty)$, we have $$\int_{[0,1]^2}f(x)W(x,y)f(y)dxdy\geq d\left(\int_{[0,1]}f(x)dx\right)^2=d||f||_1^2.$$
\end{lem}

\begin{lem}[\cite{BSW}, Lemma 2.10]\label{extend-Reiher}
Let $W$ be a $d$-locally dense graphon. Let $\omega:[0,1]\rightarrow [0,\infty)$ be a bounded measurable function.
If $H$ is a $\kn$ graph, then $$\int_{[0,1]^{v(H)}}\prod_{v\in V(H)}\omega(x_v)\prod_{uv\in E(H)}W(x_u,x_v)\prod_{v\in V(H)}dx_v\geq ||\omega||_1^{v(H)}d^{e(H)}.$$
\end{lem}
\medskip

\section{The proofs}
We devote this section to the proofs of Theorems~\ref{even sub} and \ref{regular sub}.
First we define several key notions that will be frequently utilized throughout this section.
Let $W$ be a graphon and let $s\geq 1$ be an integer.
We define the functions $W_{P_s}:[0,1]\rightarrow [0,1]$ as
\begin{equation*}
\begin{aligned}
W_{P_s}(x):&=\int_{[0,1]^s}W(x,x_1)W(x_1,x_2)...W(x_{s-1},x_s)\prod_{i=1}^{s}dx_i.
\end{aligned}
\end{equation*}
 and $W_s:[0,1]^2\rightarrow [0,1]$ as
\begin{equation*}
\begin{aligned}
W_s(x,y):&=\int_{[0,1]^{s-1}}W(x,x_1)W(x_1,x_2)...W(x_{s-2},x_{s-1})W(x_{s-1},y)\prod_{i=1}^{s-1}dx_i.\\
\end{aligned}
\end{equation*}
It is evident that $W_s$ is a graphon. We remark that it holds
\begin{equation}\label{c0}
\int_{[0,1]}W_s(x,y)dy=W_{P_s}(x).
\end{equation}
The following result is useful when computing the homomorphism density of subdivisions.

\begin{lem}\label{lem:transform}
Let $W$ be a graphon and $H^{(s)}$ be the $s$-subdivision of a graph $H$. Then it holds
\begin{equation*}
t(H^{(s)},W)=t(H,W_{s+1}).
\end{equation*}
\end{lem}

\begin{proof}
It follows from the definition of $H^{(s)}$ and $W_{s}(x,y)$ that
\begin{equation*}
t(H^{(s)},W)=\int_{[0,1]^{v(H)}}\prod_{ij\in E(H)}\left(\int_{[0,1]^s}W(x_i,y_1)\prod_{k=1}^{s-1}W(y_k,y_{k+1})W(y_s,x_j)\prod_{m=1}^sdy_m\right)\prod_{v\in V(H)}dx_v,
\end{equation*}
which equals $\int_{[0,1]^{v(H)}}\prod_{ij\in E(H)}W_{s+1}(x_i,x_j)\prod_{v\in V(H)}dx_v=t(H,W_{s+1})$, as desired.
\end{proof}

We also need a classical result in real analysis, which can be found, for example, in~\cite{Folland}.
\begin{lem}\label{closed set}
Let $E$ be a measurable set. For every $\varepsilon>0$, there exists a closed set $F\subseteq E$, such that $|E-F|<\varepsilon$.
\end{lem}

We are now ready to present the proof of Theorem~\ref{even sub}.

\begin{proof}[\bf Proof of Theorem~\ref{even sub}]
We first prove the item (a) that if $H$ is KNRS, then the $(2k-1)$-subdivision $H^{(2k-1)}$ of $H$ is Sidorenko for any integer $k\geq 1$.
Note that by Lemma~\ref{Sido-regular}, it suffices to consider any $d$-regular graphon $W$ and prove $t(H^{(2k-1)},W)\geq d^{e\big(H^{(2k-1)}\big)}$.

Recall the definition of $W_k$. We first show that $\int_{[0,1]}W_k(x,y)dy=d^k$ for every $x\in [0,1]$.
We prove this by induction on $k$. The base case $k=1$ holds trivially because $W_1=W$ is $d$-regular.
Now assume this holds for $k-1$. For every $x\in [0,1]$, it holds that
\begin{equation*}
\begin{aligned}
\int_{[0,1]}W_k(x,y)dy=&\int_{[0,1]^{k-1}}W(x,x_1)W(x_1,x_2)...W(x_{k-2},x_{k-1})\left(\int_{[0,1]}W(x_{k-1},y)dy\right)\prod_{i=1}^{k-1}dx_i\\
=&d\int_{[0,1]^{k-1}}W(x,x_1)W(x_1,x_2)...W(x_{k-2},x_{k-1})\prod_{i=1}^{k-1}dx_i\\
=&d\int_{[0,1]}W_{k-1}(x,x_{k-1})dx_{k-1}=d^k,
\end{aligned}
\end{equation*}
where the second equality holds as $W$ is $d$-regular and the last equality follows by induction.

Next we claim that $W_{2k}(x,y)$ is $d^{2k}$-locally dense.
For any measurable $A\subseteq [0,1]$, we have
\begin{equation*}
\int_{A\times A}W_{2k}(x,y)dxdy=\int_{[0,1]}\left(\int_{A\times A} W_k(x,z)W_k(z,y)dxdy\right)dz=\int_{[0,1]}\left(\int_AW_k(x,z)dx\right)^2dz.
\end{equation*}
Then using Cauchy-Schwarz inequality, it follows that
\begin{equation*}
\int_{[0,1]}\left(\int_AW_{k}(x,z)dx\right)^2dz \geq \left(\int_{[0,1]}\int_AW_{k}(x,z)dxdz\right)^2=\left(\int_{A}\left(\int_{[0,1]}W_{k}(x,z)dz\right)dx\right)^2.
\end{equation*}
Since $\int_{[0,1]}W_{k}(x,z)dz=d^k$ for any $x\in [0,1]$,
we get that $\int_{A\times A}W_{2k}(x,y)dxdy\geq \left(\int_{A}d^{k}dx\right)^2=d^{2k}|A|^2$ holds for any measurable $A\subseteq [0,1]$.
Thus, $W_{2k}(x,y)$ is $d^{2k}$-locally dense.

Using Lemma~\ref{lem:transform} and the facts that $H$ is KNRS and $W_{2k}$ is $d^{2k}$-locally dense,
we can deduce that $t(H^{(2k-1)},W)=t(H,W_{2k})\geq (d^{2k})^{e(H)}=d^{2k\cdot e(H)}=d^{e\big(H^{(2k-1)}\big)}$ for any $d$-regular graphon $W$.
By Lemma~\ref{Sido-regular}, this shows that $H^{(2k-1)}$ is Sidorenko, proving the item (a).

\bigskip

Next we prove the item (b) that if $H$ is KNRS, then the $2k$-subdivision $H^{(2k)}$ of $H$ is weakly-KNRS for any integer $k\geq 1$.
Consider any $d\in (0,1)$ and any $d$-locally dense graphon $W$.
We first show that there is a large locally dense measurable closed subset with respect to $W_{2k+1}$.

\medskip

{\bf \noindent Claim $(\star)$.} For every $\varepsilon>0$, there exists a measurable closed subset $F_\varepsilon \subseteq [0,1]$ with $|F_\varepsilon|\geq \frac{1}{2}+\phi(\varepsilon)$ such that the restricting graphon $W_{2k+1}[F_\varepsilon]$ is $\frac{d^{2k+1}}{2^{2k}}$-locally dense, where $\phi(\cdot)$ is a function such that $\phi(\varepsilon)\rightarrow 0$ when $\varepsilon\rightarrow 0$.

\begin{proof}[Proof of Claim $(\star)$.]
Let $D=\{x\in [0,1]:W_{P_k}(x)\geq (\frac{d}{2})^k\}$ and $\bar{D}:=[0,1] \setminus D$ be the complement of $D$ in $[0,1]$. By Lemma~\ref{closed set}, for every $\varepsilon>0$, there exist closed sets $F_\varepsilon$ and $E_\varepsilon$ such that $F_\varepsilon\subseteq D$ and $E_\varepsilon\subseteq \bar{D}$ with $|D-F_\varepsilon|<\varepsilon$ and $|\bar{D}-E_\varepsilon|<\varepsilon$. We will proceed to show that the closed subset $F_\varepsilon$ is the desired subset.
First we demonstrate that $|F_\varepsilon|\geq \frac{1}{2}+\phi(\varepsilon)$ for some function $\phi(\cdot)$ such that $\phi(\varepsilon)\rightarrow 0$ when $\varepsilon\rightarrow 0$.
Note that $|\bar{D}|=1-|D|$.
By the definition, we have that
\begin{equation}\label{c1}
\int_{\bar{D}}W_{P_k}(x)dx<\left(\frac{d}{2}\right)^k\cdot (1-|D|).
\end{equation}
On the other hand, it holds that
\begin{equation}\label{c2}
\begin{aligned}
\int_{\bar{D}}W_{P_k}(x)dx &\geq \int_{\bar{D}^{k+1}}\prod_{ij\in E(P_k)}W(x_i,x_j)\prod_{v\in V(P_k)}dx_v\\
&\geq \int_{E_\varepsilon^{k+1}}\prod_{ij\in E(P_k)}W(x_i,x_j)\prod_{v\in V(P_k)}dx_v\\
&=t(P_k,W[E_\varepsilon])\cdot |E_\varepsilon|^{k+1}.
\end{aligned}
\end{equation}
where the last equality follows from \eqref{W[A]'}.
Using the well-known fact that every path $P_k$ is Sidorenko and by Lemma~\ref{sub locally dense} that the restricting graphon $W[E_\varepsilon]$ is $d$-locally dense, we have
\begin{equation}\label{c3}
t(P_k,W[E_\varepsilon])\geq t(K_2,W[E_\varepsilon])^k\geq d^k.
\end{equation}
Combining with inequalities \eqref{c1}, \eqref{c2}, \eqref{c3} and $|\bar{A}|-|E_\varepsilon|<\varepsilon$,
we deduce that $$\left(\frac{d}{2}\right)^k(1-|A|)> d^k\cdot |E_\varepsilon|^{k+1}\geq d^k\cdot (1-|D|-\varepsilon)^{k+1},$$ which implies that $|D|\geq \frac{1}{2}+\varphi(\varepsilon),$ where $\varphi(\cdot)$ is some function such that $\varphi(\varepsilon)\rightarrow 0$ when $\varepsilon\rightarrow 0$. Note that by Lemma~\ref{closed set}, we have that $|D|-|F_\varepsilon|<\varepsilon$, which implies that $|F_\varepsilon|>|D|-\varepsilon\geq\frac{1}{2}+\varphi(\varepsilon)-\varepsilon.$ Let $\phi(\varepsilon)$ be $\varphi(\varepsilon)-\varepsilon$. Then we have done.

It remains to show that $W_{2k+1}[F_\varepsilon]$ is $\frac{d^{2k+1}}{2^{2k}}$-locally dense.
Let $F$ be defined by setting the measurable closed set $A=F_\varepsilon$ in \eqref{f(z)}.
Consider any measurable $B'\subseteq [0,1]$ and define $B$ by \eqref{equ:B}.
Then $B\subseteq F_\varepsilon$ and $|B'|=\frac{|B|}{|F_\varepsilon|}.$
By Lemma~\ref{F(B|A|)},
\begin{equation}\label{d0}
\int_{B'\times B'}W_{2k+1}[F_\varepsilon](x,y)dxdy=\frac{1}{|F_\varepsilon|^2}\int_{B\times B}W_{2k+1}(x,y)dxdy.
\end{equation}
Since $W$ is $d$-locally dense, by setting $f(z)=\int_{B}W_k(x,z)dx$, it follows from Lemma~\ref{Reiher} that
\begin{equation}\label{d1}
\begin{aligned}
\int_{B\times B}W_{2k+1}(x,y)dxdy=&\int_{[0,1] \times [0,1]}W(z,w)\left(\int_{B}W_k(x,z)dx\right)\left(\int_{B}W_k(y,w)dy\right)dzdw\\
=&\int_{[0,1] \times [0,1]}W(z,w)f(z)f(w)dzdw\geq d||f||_1^2.
\end{aligned}
\end{equation}
Moreover, we can deduce that
\begin{equation}\label{d2}
||f||_1^2=\left(\int_{B}\left(\int_{[0,1]}W_k(x,z)dz\right)dx \right)^2=\left(\int_{B}W_{P_k}(x)dx\right)^2\geq |B|^2\cdot (d/2)^{2k},
\end{equation}
where the last inequality follows from the fact that $B\subseteq F_\varepsilon \subseteq A$ and the definition of $A$.
Recall that $|B'|=\frac{|B|}{|F_\varepsilon|}.$
Combining (\ref{d0}) with (\ref{d1}) and (\ref{d2}), we obtain that $$\int_{B'\times B'}W_{2k+1}[F_\varepsilon](x,y)dxdy\geq \frac{d^{2k+1}}{2^{2k}}\frac{|B|^2}{|F_\varepsilon|^2}=\frac{d^{2k+1}}{2^{2k}}|B'|^2$$ holds for any measurable subset $B'\subseteq [0,1]$.
This shows that $W_{2k+1}[F_\varepsilon]$ is indeed $\frac{d^{2k+1}}{2^{2k}}$-locally dense, completing the proof of Claim $(\star)$.
\end{proof}

Since $H$ is KNRS and $W_{2k+1}[F_\varepsilon]$ is $\frac{d^{2k+1}}{2^{2k}}$-locally dense, by Lemma~\ref{lem:transform} and \eqref{W[A]} we have
\begin{equation*}
\begin{aligned}
t(H^{(2k)},W)&=t(H,W_{2k+1})\geq |F_\varepsilon|^{v(H)} t(H,W_{2k+1}[F_\varepsilon])\geq \left(1/2+\phi(\varepsilon)\right)^{v(H)}\cdot \left(\frac{d^{(2k+1)}}{2^{2k}}\right)^{e(H)}.
\end{aligned}
\end{equation*}
holds for every $d$-locally dense graphon $W$ and every $\varepsilon>0$. By letting $\varepsilon\rightarrow 0$, we have that $$t(H^{(2k)},W)\geq c_H\cdot d^{(2k+1)e(H)}=c_H\cdot d^{e\big(H^{(2k)}\big)},$$ where $c_H=(\frac{1}{2})^{v(H)+2ke(H)}$ is an absolute constant.
By Lemma~\ref{weakly-knrs-limits}, $H^{(2k)}$ is weakly-KNRS, finishing the proof of Theorem~\ref{even sub}.
\end{proof}

In the rest of this section, we prove Theorem~\ref{regular sub}.

\begin{proof}[\bf Proof of Theorem~\ref{regular sub}]
Consider a graph $H$ which is KNRS and regular, say every vertex $v\in V(H)$ has degree $\Delta\geq 1$.
We want to show that the $2k$-subdivision $H^{(2k)}$ of $H$ is KNRS for any integer $k\geq 1$.
Let $d\in (0,1)$ and let $W$ be any $d$-locally dense graphon.
By Lemma~\ref{graphon form}, it suffices for us to prove that $t(H^{(2k)},W)\geq d^{(2k+1)e(H)},$ where $e(H^{(2k)})=(2k+1)e(H)$.

We proceed to prove a series of claims.
The first claim asserts that the set consisting of $x\in [0,1]$ with $W_{P_k}(x)=0$ must be a zero-measure set.

\begin{claim}\label{zero-measure}
For integers $k\geq 1$, let $B_k=\{x\in [0,1]:~ W_{P_k}(x)=0\}$.
Then it holds that $|B_k|=0$.
\end{claim}

\begin{proof}
Suppose for a contradiction that $|B_k|>0$. Then by Lemma~\ref{closed set}, for $\varepsilon>0$, there exists a close set $E\subseteq B_k$ with $|E|>0$ and $|B_k-E|<\varepsilon$. We can estimate as follows:
\begin{equation*}
\begin{aligned}
\int_{B_k}W_{P_k}(x)dx&\geq \int_{{B_k}^{k+1}}\prod_{ij\in E(P_k)}W(x_i,x_j)\prod_{v\in V(P_k)}dx_v\\
&\geq \int_{{E}^{k+1}}\prod_{ij\in E(P_k)}W(x_i,x_j)\prod_{v\in V(P_k)}dx_v\\
&=|E|^{k+1}\cdot t(P_k,W[E])\geq |E|^{k+1}\cdot t(K_2,W[E])^k\geq |E|^{k+1}\cdot d^k>0,
\end{aligned}
\end{equation*}
where the first equality follows from (\ref{W[A]'}), the second inequality holds because $P_k$ is Sidorenko,
and the last inequality can be deduced by Lemma~\ref{sub locally dense} that $W[E]$ is $d$-locally dense.
On the other hand, by the definition of $B_k$ it holds that $\int_{B_k}W_{P_k}(x)dx=0$, a contradiction.
\end{proof}

We define the function $W_{2k+1}'(x,y): [0,1]^2 \to [0,1]$ as follows:
if $W_{P_k}(x)W_{P_k}(y)=0$ for $x,y\in [0,1]$, define $W_{2k+1}'(x,y)=0$;
otherwise $W_{P_k}(x)W_{P_k}(y)>0$, define $$W_{2k+1}'(x,y):=\frac{W_{2k+1}(x,y)}{W_{P_k}(x)W_{P_k}(y)}.$$

\begin{claim}
$W_{2k+1}'$ is a graphon.
\end{claim}

\begin{proof}
It is obvious that $W_{2k+1}'$ is measurable, $W_{2k+1}'(x,y)=W_{2k+1}'(y,x)$, and $W_{2k+1}'(x,y)\geq 0$ for all $x,y\in [0,1]$.
Thus, it remains to prove that $W_{2k+1}'(x,y)\leq 1$.
Note that for $x,y\in [0,1]$ with $W_{P_k}(x)W_{P_k}(y)=0$, we have $W_{2k+1}'(x,y)=0$.
Otherwise, it holds that
\begin{equation*}
\begin{aligned}
 W_{2k+1}'(x,y)&=\frac{W_{2k+1}(x,y)}{W_{P_k}(x)W_{P_k}(y)}=\frac{\int_{[0,1]^2}W_{k}(x,z)W(z,w)W_{k}(y,w)dzdw}{W_{P_k}(x)W_{P_k}(y)}\\
 &\leq \frac{\int_{[0,1]^2}W_{k}(x,z)W_{k}(y,w)dzdw}{W_{P_k}(x)W_{P_k}(y)}=1,
\end{aligned}
\end{equation*}
where the last inequality follows from (\ref{c0}).
So indeed $W_{2k+1}'$ is a graphon.
\end{proof}

Next we define an auxiliary function $U_{k}(x,y):[0,1]^2 \rightarrow [0,\infty)$ as follows:
if $W_{P_k}(x)=0$ for $x\in [0,1]$, then for any $y\in [0,1]$ we define $U_{k}(x,y)=0$;
otherwise we define
$$U_{k}(x,y):=\frac{W_k(x,y)}{W_{P_k}(x)}.$$
Let $B_k$ be the subset defined in Claim~\ref{zero-measure}.
For any $A\subseteq [0,1]$, let $\hat{A}_k$ be the subset $A \setminus B_k$.
By Claim~\ref{zero-measure}, $B_k$ is a zero-measure set, which yields that $|\hat{A}_k|=|A|$ and thus
\begin{equation}\label{e2}
\begin{aligned}
 \int_{A\times A}W_{2k+1}'(x,y)dx=\int_{\hat{A}_k\times \hat{A}_k}W_{2k+1}'(x,y)dx.
\end{aligned}
\end{equation}
For any $x\in \hat{A}_k$, it holds that $W_{P_k}(x)>0$ and thus $\int_{[0,1]}\frac{W_k(x,z)}{W_{P_k}(x)}dz=1$.
Hence, we have
\begin{equation}\label{e3}
\begin{aligned}
\int_{[0,1]}\left(\int_{\hat{A}_k}U_k(x,y)dx\right)dz &=\int_{[0,1]}\left(\int_{\hat{A}_k}\frac{W_k(x,z)}{W_{P_k}(x)}dx\right)dz\\
&=\int_{\hat{A}_k}\left(\int_{[0,1]}\frac{W_k(x,z)}{W_{P_k}(x)}dz\right)dx=\int_{\hat{A}_k}1dx=|\hat{A}_k|=|A|.
\end{aligned}
\end{equation}

\begin{claim}\label{claim4}
The graphon $W_{2k+1}'$ is $d$-locally dense.
\end{claim}

\begin{proof}
For any measurable subset $A\subseteq [0,1]$, we again use $\hat{A}_k$ to represent the subset of $A$ consisting of all $x\in A$ with $W_{P_k}(x)>0$. Using \eqref{e2}, we obtain that
\begin{equation*}\label{claim4(1)}
\begin{aligned}
\int_{A\times A}W_{2k+1}'(x,y)dxdy&=\int_{\hat{A}_k\times \hat{A}_k}W_{2k+1}'(x,y)dxdy\\
&=\int_{\hat{A}_k\times \hat{A}_k}\frac{W_{2k+1}(x,y)}{W_{P_k}(x)W_{P_k}(y)}dxdy\\
&=\int_{\hat{A}_k\times \hat{A}_k}\left(\int_{[0,1]^2}\frac{W_k(x,z)W(z,w)W_k(w,y)}{W_{P_k}(x)W_{P_k}(y)}dzdw\right)dxdy\\
&=\int_{[0,1]^2}W(z,w)\left(\int_{\hat{A}_k}\frac{W_k(x,z)}{W_{P_k}(x)}dx\right)\left(\int_{\hat{A}_k}\frac{W_k(w,y)}{W_{P_k}(y)}dy\right)dzdw\\
&=\int_{[0,1]^2}W(z,w)\left(\int_{\hat{A}_k}U_{k}(x,z)dx\right)\left(\int_{\hat{A}_k}U_k(y,w)dy\right)dzdw.
\end{aligned}
\end{equation*}
Applying Lemma~\ref{Reiher} with $f(z):=\int_{\hat{A}_k}U_k(x,z)dx$, we deduce that
\begin{equation*}\label{claim4(2)}
\begin{aligned}
\int_{A\times A}W_{2k+1}'(x,y)dxdy&=\int_{[0,1]^2}W(z,w)f(z)f(w)dzdw\geq d||f||_1^2\\
&=d\left(\int_{[0,1]}\left(\int_{\hat{A}_k}U_k(x,z)dx\right)dz\right)^2=d|A|^2,
\end{aligned}
\end{equation*}
where the last equality holds by \eqref{e3}.
This yields that $W_{2k+1}'$ is a $d$-locally dense graphon.
\end{proof}

By Claim~\ref{zero-measure}, we have $W_{2k+1}(x,y)=W_{2k+1}'(x,y)W_{P_k}(x)W_{P_k}(y)$ for all $(x,y)\in [0,1]^2$ except for a zero-measure set.
Hence, since $H$ is $\Delta$-regular, we can obtain that
\begin{equation*}
\begin{aligned}
t(H^{(2k)},W)&=t(H,W_{2k+1})=\int_{[0,1]^{v(H)}}\prod_{ij\in E(H)}W_{2k+1}(x_i,x_j)\prod_{v\in V(H)}dx_v\\
&=\int_{[0,1]^{v(H)}}\prod_{ij\in E(H)}W'_{2k+1}(x_i,x_j)\prod_{v\in V(H)}(W_{P_k}(x_v))^{\Delta}\prod_{v\in V(H)}dx_v.
\end{aligned}
\end{equation*}
From Claim~\ref{claim4}, we have that $W_{2k+1}'$ is $d$-locally dense.
Now applying Lemma~\ref{extend-Reiher} with $\omega(x)=(W_{P_k}(x))^{\Delta}$ for $x\in [0,1]$, we can derive that
\begin{equation}\label{e4}
\begin{aligned}
t(H^{(2k)},W)&=\int_{[0,1]^{v(H)}}\prod_{ij\in E(H)}W'_{2k+1}(x_i,x_j)\prod_{v\in V(H)}\omega(x_v)\prod_{v\in V(H)}dx_v\\
&\geq d^{e(H)}\left(\int_{[0,1]}\omega(x) dx\right)^{v(H)}\geq d^{e(H)}\left(\int_{[0,1]}W_{P_k}(x)dx\right)^{\Delta\cdot v(H)},
\end{aligned}
\end{equation}
where the last inequality follows from Jensen's inequality applied to the function $x\mapsto x^\Delta$, where $\Delta\geq 1$.
Finally, using the facts that $P_k$ is Sidorenko and $W$ is $d$-locally dense, we have
\begin{equation}\label{e5}
\int_{[0,1]}W_{P_k}(x)dx=t(P_k,W)\geq t(K_2,W)^{k}\geq d^{k}.
\end{equation}
Combining (\ref{e4}) with(\ref{e5}) and using the fact that $2e(H)=\Delta\cdot v(H)$,
we obtain that $$t(H^{(2k)},W)\geq d^{e(H)+k\Delta v(H)}=d^{(2k+1)e(H)}$$
holds for any $d$-locally dense graphon $W$.
By Lemma~\ref{graphon form}, we prove that $H^{(2k)}$ is KNRS.
\end{proof}

\section{Concluding remarks}
We prove that if $H$ is a $\kn$ graph, then its odd-subdivisions $H^{(2k-1)}$ are Sidorenko, and its even-subdivisions $H^{(2k)}$ are weakly-$\kn$.
Despite our suspicion that all $H^{(2k)}$ are in fact $\kn$ graphs,
we are unable to significantly improve the leading coefficient $c_H$ in the proof of Theorem~\ref{even sub}. 
We remark that the usual tensor power trick does not work here (for a counterexample, see Proposition 5.1 in \cite{BSW}).
However, it is worth noting that this indeed holds true when we assume the additional condition that $H$ is regular, as demonstrated in Theorem~\ref{regular sub}.

If both Conjecture~\ref{knrs} and Conjecture~\ref{regular-KNRS} are proven true, the equivalence of $\kn$ graphs, weakly-$\kn$ graphs, and regular-$\kn$ graphs becomes evident.
However, since both conjectures pose significant challenges,
it would be interesting to explore alternative approaches (such as employing deductive reasoning) to investigate the equivalence between these concepts, independent of Conjecture~\ref{knrs} and Conjecture~\ref{regular-KNRS}.
One intriguing question arises: Is it true that all weakly-$\kn$ graphs are $\kn$ graphs?
This question arises particularly in the context of even-subdivisions of a $\kn$ graph, as discussed in the preceding paragraph.

\appendix
\section{Proof of Lemma~\ref{F(B|A|)}}\label{app}
We restate Lemma~\ref{F(B|A|)} as follows.

\begin{lem}
Let $W$ be any graphon and let $A, B'$ be measurable subsets in $[0,1]$ with $|A|>0$.
Define the subset $B$ by \eqref{equ:B}. If $A$ is closed, then it holds that
$$\int_{B'\times B'}W[A](x,y)dxdy=\frac{1}{|A|^2}\int_{B\times B}W(x,y)dxdy.$$
\end{lem}

\begin{proof}
For graphons $W$ and $U$, the {\it Hadamard product} of $W$ and $U$ is defined by $W\bigodot U(x,y)=W(x,y)\cdot U(x,y)$.
Note that $W\bigodot U$ is also a graphon.

By the definition of $B$, we have that $\mathbbm{1}_{B'}(x)=\mathbbm{1}_{B}(F(x\cdot |A|))$ for all $x\in [0,1]$. We define a new graphon $U$ as $U(x,y)=\mathbbm{1}_B(x)\mathbbm{1}_B(y)$ for $(x,y)\in [0,1]^2$. Then
\begin{equation}\label{F(B|A|)1}
U[A](x,y)=U(F(x\cdot|A|),F(y\cdot |A|))=\mathbbm{1}_B(F(x\cdot|A|))\mathbbm{1}_B(F(y\cdot |A|))=\mathbbm{1}_{B'}(x)\mathbbm{1}_{B'}(y)
\end{equation}
and
\begin{equation}\label{F(B|A|)2}
\begin{aligned}
W[A](x,y)\cdot U[A](x,y)&=W\big(F(x\cdot |A|),F(y\cdot |A|)\big)\cdot U\big(F(x\cdot |A|),F(y\cdot |A|)\big)\\
&=(W\bigodot U)(F(x\cdot |A|),F(y\cdot |A|))=(W\bigodot U)[A](x,y)
\end{aligned}
\end{equation}
hold for all $(x,y)\in [0,1]^2$.
Combining with \eqref{F(B|A|)1} and \eqref{F(B|A|)2}, it follows that
\begin{equation*}
\begin{aligned}
\int_{B'\times B'}W[A](x,y)dxdy=&\int_{[0,1]^2}W[A](x,y)\mathbbm{1}_{B'}(x)\mathbbm{1}_{B'}(y)dxdy\\
=&\int_{[0,1]^2}W[A](x,y)\cdot U[A](x,y)dxdy=t\big(K_2,(W\bigodot U)[A]\big).
\end{aligned}
\end{equation*}
Using \eqref{W[A]'} and the fact $B\subseteq A$, we derive that $t(K_2,(W\bigodot U)[A])$ is equal to
\begin{equation*}
\begin{aligned}
&\frac{1}{|A|^2}\int_{[0,1]^2}(W\bigodot U)(x,y)\mathbbm{1}_A(x)\mathbbm{1}_A(y)dxdy=\frac{1}{|A|^2}\int_{[0,1]^2}W(x,y)U(x,y)\mathbbm{1}_A(x)\mathbbm{1}_A(y)dxdy\\
=&\frac{1}{|A|^2}\int_{[0,1]^2}W(x,y)\mathbbm{1}_B(x)\mathbbm{1}_B(y)dxdy=\frac{1}{|A|^2}\int_{B\times B}W(x,y)dxdy.
\end{aligned}
\end{equation*}
Combining these facts, we get $\int_{B'\times B'}W[A](x,y)dxdy=\frac{1}{|A|^2}\int_{B\times B}W(x,y)dxdy$, as desired.
\end{proof}

\bigskip

{\it E-mail address:} mathsch@mail.ustc.edu.cn

\medskip

{\it E-mail address:} lyp\_inustc@mail.ustc.edu.cn

\medskip

{\it E-mail address:} jiema@ustc.edu.cn

\end{document}